\documentclass[11pt]{amsart}

\usepackage{cmap}
\usepackage[T2A]{fontenc}
\usepackage[utf8]{inputenc}
\usepackage[english]{babel}
\usepackage{amsfonts,amssymb}
\usepackage{hyperref}
\usepackage{faktor}
\usepackage{mathrsfs} 
\usepackage{stmaryrd}
\usepackage{epsfig} 
\usepackage[left=2.1cm,right=2.0cm,top=2.0cm,bottom=2.0cm]{geometry}

\usepackage[all,cmtip]{xy}
\usepackage{capt-of} 

\usepackage{enumitem} 

\newcommand{\mytitle}{Explicit equations for exterior square of the general linear group}
\title{\mytitle}

\author{Roman~Lubkov}
\thanks{Theorems~\ref{Congr}--\ref{Tutor} are done by the first author and supported by the Russian Science Foundation, grant №17-11-01261.}
\address[Roman~Lubkov]{Department of Mathematics and Mechanics, St.~Petersburg State University}
\email{RomanLubkov@yandex.ru}

\author{Ilia~Nekrasov}
\thanks{Theorems~\ref{Plu}--\ref{again} are done by the second author and supported by the Russian Science Foundation, grant №16-11-10200. Also the second author is thankful to ``Native towns'', a social investment program of PJSC ``Gazprom Neft''.}
\address[Ilia~Nekrasov]{Chebyshev Laboratory, St.~Petersburg State University}
\email{geometr.nekrasov@yandex.ru}

\keywords{General linear group, weight diagrams, exterior square}

\subjclass{20G35}

\date{}
\DeclareMathOperator{\E}{E}
\DeclareMathOperator{\M}{M}
\DeclareMathOperator{\SL}{SL}
\DeclareMathOperator{\GL}{GL}
\DeclareMathOperator{\sign}{sgn}
\DeclareMathOperator{\Gr}{Gr}
\DeclareMathOperator{\Sym}{Sym}
\DeclareMathOperator{\Plu}{Pl\ddot{u}}
\DeclareMathOperator{\Pic}{Pic}

\renewcommand{\trianglelefteq}{\trianglelefteqslant}
\renewcommand{\leq}{\leqslant}
\renewcommand{\geq}{\geqslant}

\newcommand{\bigwedgem}[1]{\mathord{\raisebox{2pt}
{\hbox{$\scriptstyle{\bigwedge^{\!#1}}$}}}}
\newcommand\blank{\mathord{\hbox to 1.5ex{\hrulefill}}\,}

\theoremstyle{plain}
\newtheorem{theorem}{Theorem}
\newtheorem{proposition}{Proposition}
\newtheorem{lemma}{Lemma}
\newtheorem{corollary}{Corollary}
\theoremstyle{remark}
\newtheorem*{remark}{Remark}

\begin{document}

\begin{abstract}
We present several explicit systems of equations defining exterior square of the general linear group $\bigwedgem{2}\GL_{n}$ as an affine group scheme. Algebraic ingredients of the equations, exterior numbers, are translated into the language of weight diagrams corresponding to Lie groups of type $A_{n-1}$ in representation with the highest weight $\varpi_{2}$.
\end{abstract}

\maketitle

\section*{Introduction}\label{intro}

The starting point of the present paper is the following problem: to describe overgroups of elementary subgroups of Chevalley groups. Classically, one of the key steps in a proof of so-called \textit{standard description} uses explicit equations defining a Chevalley groups. In~\cite{LubNekover1} this problem is partially solved for groups of type $A_{n}$ in representation with the highest weight $\varpi_{m}$ ($m \leq n+1$). But the technique of explicit equations is replaced by methods of representation theory.

Using stabilizing of quadratic forms in~\cite{PetOvergr,VP-EOeven,VP-Ep} equations on Chevalley groups were obtained. Also, in~\cite{VavLuzgE6} this technique was developed for stabilizing of the cubic form.

Following the paradigm of the mentioned papers, authors construct several explicit systems of equations defining an affine group scheme $\bigwedgem{2}\GL_{n}$. This case corresponds to a group of type $A_{n-1}$ in the second fundamental representation.

Let us remark that methods of representation theory for a general exterior power use structural results for plethysms $\Sym^k( \bigwedgem{m} \mathbb{C}^n)$ with arbitrary natural numbers $k, m, n$. But in the present paper, we use only one fact: the group $\bigwedgem{2}\GL_{n}$ over a commutative ring $R$ preserves an ideal $\Plu$ generated by the Pl\"ucker relations. The last argument can be interpreted as follows. The algebraic variety $\Gr_{2,n}$ is stabilized under the action of the algebraic group $\bigwedgem{2}\GL_{n}$. Therefore, the present method
\begin{itemize}
\item is more elementary and transparent;
\item allows to understand inner structure of the group scheme $\bigwedgem{2}\GL_{n}$ in more explicit way.
\end{itemize}

For a general exterior power the Pl\"ucker ideal structure is much more complicated. Hence, it is not possible to generalize the results of the present paper. Structure of the Pl\"ucker ideal and its higher syzygies in the case of the exterior square  are described in~\cite{Khorosh} in details.

The present paper is organized as follows. In Section~\ref{prelim} we recall basic notation and set all main results pertaining to the group scheme  $\bigwedgem{2}\GL_{n}$. In the next Section we state and prove the main result: explicitly described equations for the scheme $\bigwedgem{2}\GL_{n}$. The last Section is devoted to translation of the algebraic structure of equations into weight diagram terms.

\section{Preliminaries}\label{prelim}
\noindent

Our notation is for the most part fairly standard in Chevalley group theory and coincides with the notation in~\cite{StepVavDecomp,VP-EOeven,VP-Ep,VP-EOodd}. We recall all necessary notion to read the present paper independently.

First, let $G$ be an arbitrary group. By a commutator of two elements we always understand \textit{the left-normed} commutator $[x,y]=xyx^{-1}y^{-1}$, where $x,y\in G$. Multiple commutators are also left-normed; in particular, $[x,y,z]=[[x,y],z]$. By ${}^xy=xyx^{-1}$ we denote \textit{the left conjugates} of $y$ by $x$. Similarly, by $y^x=x^{-1}yx$ we denote \textit{the right conjugates} of $y$ by $x$. 

For a subset $X\subseteq G$, we denote by $\langle X\rangle$ the subgroup it generates. The notation $H\leq G$ means that $H$ is a subgroup in $G$, while the notation $H\trianglelefteq G$ means that $H$ is a normal subgroup in $G$. For $H\leq G$, we denote by $\langle X\rangle^H$ the smallest subgroup in $G$ containing $X$ and normalized by $H$. For two groups $F,H\leq G$, we denote by $[F,H]$ their mutual commutator: $[F,H]=\langle [f,g], \text{ where } f\in F, h\in H\rangle.$

Also, we need some elementary ring theory notation. Let $R$ be an arbitrary associative ring with 1. By default, it is assumed to be commutative. By an ideal $I$ of the ring $R$ we understand \textit{the two-sided ideal} and this is denoted by $I\trianglelefteq R$. As usual, let $R^{*}$ be the multiplicative group of the ring $R$. Let $M(m, n, R)$ be the $R$-bimodule of $(m \times n)$-matrices with entries in $R$, and let $M(n, R) = M(n, n, R)$ be the full matrix ring of degree $n$ over $R$. By $\GL_n(R)=\M(n,R)^{*}$ we denote the general linear group. As usual, $a_{i,j}$ denotes the entry of a matrix $a$ at the position  $(i,j)$, where $1\leq i,j\leq n$. 

By $[n]$ we denote a set $\{1,2,\ldots, n\}$ and by $\bigwedgem{m}[n]$ we denote the exterior power of the set $[n]$. Elements of $\bigwedgem{m}[n]$ are unordered\footnote{In the sequel, we arrange them in the ascending order.} subsets $I\subseteq [n]$ of cardinality $m$ without repeating entries:
$$\bigwedgem{m}[n] = \{ (i_{1}, i_{2}, \ldots , i_{m})\; |\; i_{j} \in [n], i_{j} \neq i_{l} \}.$$  

Let $I = \{i_{1} < \ldots < i_{v}\} \in \bigwedgem{v}[n]$, $J = \{j_{1} < \ldots < j_{u}\} \in \bigwedgem{u}[n]$; then by $\sign(I,J)$ we denote the sign of the permutation $(i_{1}, i_{2}, \dots, i_{v} , j_{1}, \dots , j_{u})$.

Also, for a matrix $a \in \GL_{n}(R)$ and for sets $I$, $J$ from $\bigwedgem{m}[n]$, define a minor $M_{I}^{J}(a)$ of the matrix $a$ as follows. $M_{I}^{J}(a)$ equals the determinant of a submatrix formed by rows from the set $I$ and columns from the set $J$. 

\subsection{Exterior power of the general linear group}

In this Section we describe the exterior square of the general linear group over a commutative ring $R$; details can be found in~\cite{LubNekover1, VavPere}. In the sequel, we keep the rank $n$ of the base general linear group $\GL_{n}$.

Let $R^n$ be a right $R$--module with the basis $\{e_1,\ldots,e_n\}$. Consider the standard action of the group $\GL_{n}(R)$ on $R^n$. Define an exterior square of an $R$--module as follows. Basis of this module is all exterior products $e_i\wedge e_j$, $1~\leq~i~\neq~j~\leq~n$ and $e_{i} \wedge e_{j} = -e_{j} \wedge e_{i}$. The rank of this module equals the binomial coefficient $\binom{n}{2}$. We denote this number by $N$.

Now, we define an action of the group $\GL_{n}(R)$ on $\bigwedgem{2}(R^n)$. Firstly, we define this action on elements of the basis by the rule
$$g(e_i\wedge e_j):=(ge_i)\wedge(ge_j) \text{ for any } g\in \GL_{n}(R) \text{ and } 1\leq i\neq j\leq n.$$
Secondly, we extend this action by linearity to the whole module $\bigwedgem{2}(R^n)$. Finally, using this action we define a subgroup $\bigwedgem{2}\big(\GL_{n}(R)\big)$ of the general linear group $\GL_{N}(R)$.

In other words, let us consider the Cauchy--Binet homomorphism
$$\bigwedgem{2}:\GL_{n}(R)\longrightarrow \GL_{N}(R),$$
taking each matrix $x\in \GL_{n}(R)$ to the matrix $\bigwedgem{2}(x)\in \GL_{N}(R)$. Elements of $\bigwedgem{2}(x)\in \GL_{N}(R)$ are all second order minors of the matrix $x$. Then the group $\bigwedgem{2}\big(\GL_{n}(R)\big)$ is an image of the general linear group under the Cauchy--Binet homomorphism. It is natural to index elements of the matrix $\bigwedgem{2}(x)$ by pairs of elements of the set $\bigwedgem{2}[n]$:
$$\left(\bigwedgem{2}(x)\right)_{I,J}=\left(\bigwedgem{2}(x)\right)_{(i_1,i_2),(j_1,j_2)}=M_{i_1,i_2}^{j_1,j_2}(x) = x_{i_1, j_1}\cdot x_{i_2, j_2} - x_{i_1, j_2}\cdot x_{i_2, j_1}.$$

The last can be generalized to the case of an arbitrary $m$-th exterior power (we assume that $m \leq n$). Therefore, the [abstract] group $\bigwedgem{m}(\GL_{n}(R))$ is well defined for arbitrary commutative ring $R$. In a general case this group is not a group of points of any algebraic group. Hence, by $\bigwedgem{m}\GL_{n}$ we denote the corresponding algebraic group. It is the [Zariski] closure of all $\bigwedgem{m}(\GL_{n}(R))$. Let us remark that this group $\bigwedgem{m}\GL_{n}$ is a subgroup of the algebraic group $\GL_{\binom{n}{m}}$. It equals $\GL_{n}$ in representation with the highest weight $\varpi_{m}$.

We stress that to obtain equations on the scheme $\bigwedgem{m}\GL_{n}$ it is sufficient to obtain equations on the group $\bigwedgem{m}\GL_{n}(\mathbb{C})$. It is true due to the following fact. Obtained equations would be defined over $\mathbb{Z}$. Consequently, they determine the group as a scheme over $\mathrm{Spec}(\mathbb{Z})$.

For an arbitrary ring $R$ a group of $R$-points of the scheme $\bigwedgem{m}\GL_{n}(R)$ is strictly greater than the corresponding group-theoretic image and is strictly lesser than a group of $R$-points of the ambient group scheme:

\begin{align*}
\bigwedgem{m} \big(\GL_{n}(R)\big) &< \bigwedgem{m}\GL_{n}(R) < \GL_{\binom{n}{m}}(R).
\end{align*}

We refer the reader to~\cite{VavPere} for more precise results about the difference between the last three groups. We formulate the main result in the following theorem

\begin{theorem}\label{Pere}
In the previous notation: 
\begin{enumerate}[label=(\arabic*), ref=\thetheorem(\arabic*)]
\item For any $n \geq 4$, we have an isomorphism of affine group schemes 
$$\bigwedgem{2}\GL_n \cong \GL_{n}/\mu_{2}.$$
\item The quotient group $\faktor{\bigwedgem{2}\GL_n (R)}{\bigwedgem{2}(\GL(n,R))}$ contains a copy of the group $\faktor{R^{*}}{R^{*2}}$, and the further quotient modulo this copy is a subgroup of the Picard group $\Pic(R)$, consisting of invertible modules $P$ over $R$ such that 
$$P^{\otimes 2}=R,\quad P^n=R^n.$$
\item\label{Extens} For any element $x$ of the group $\bigwedgem{2}\GL_n (R)$ there exists a finite extension $R'/R$ such that for some $y \in \GL_{n}(R')$, we have 
$$x = \bigwedgem{2}(y).$$
\end{enumerate}
\end{theorem}

\subsection{Elementary group and its exterior powers}
The elementary group $\E(n, R)$ plays a special role among subgroups of the general linear group. As always, elements of this group are characterized by the following fact. Technically cumbersome calculations for the generic element of the group $\GL_{n}$ are much more transparent and at the same time non-trivial, as, for example, for torus elements.

Recall that $e$ denotes the identity matrix and $e_{i,j}$ denotes the standard matrix unit, i.\,e., the matrix that has 1 at the position $(i,j)$ and zeros elsewhere.

By $t_{i,j}(\xi)$ we denote an elementary transvection, i.\,e., a matrix of the form $t_{i,j}(\xi)=e+\xi e_{i,j}$, $1\leq i\neq j\leq n$, $\xi\in R$. In the sequel, we use (without any special reference) standard relations~\cite{StepVavDecomp} among elementary transvections such as
\begin{enumerate}
\item the additivity:
$$t_{i,j}(\xi)t_{i,j}(\zeta)=t_{i,j}(\xi+\zeta).$$
\item the Chevalley commutator formula:
$$[t_{i,j}(\xi),t_{h,k}(\zeta)]=
\begin{cases}
e,& \text{ if } j\neq h, i\neq k,\\
t_{i,k}(\xi\zeta),& \text{ if } j=h, i\neq k,\\
t_{h,j}(-\zeta\xi),& \text{ if } j\neq h, i=k.
\end{cases}$$
\end{enumerate}

The subgroup $\E(n,R)\leq \GL_{n}(R)$ generated by all elementary transvections, is called the \textit{(absolute)} elementary group:
$$\E(n,R)=\langle t_{i,j}(\xi), 1\leq i\neq j\leq n, \xi\in R\rangle.$$

It is well known (due to Andrei Suslin~\cite{SuslinSerreConj}) that the elementary group is normal in the general linear group $\GL_{n}(R)$ for $n \geq 3$. As a straightforward corollary, we have the following result.

\begin{lemma}\label{SuslinFor2}
The image of the elementary group is normal in the image of the general linear group under the exterior square homomorphism:
$$\bigwedgem{2}\left(\E(n,R)\right)\trianglelefteq \bigwedgem{2}(\GL_{n}(R)).$$
\end{lemma}

Let us consider a structure of the group $\bigwedgem{2}\E(n,R)$ in details. The following proposition can be obtained by the very definition of $\bigwedgem{2}\big(\GL_{n}(R)\big)$.

\begin{proposition} \label{ImageOfTransvFor2}
Let $t_{i,j}(\xi)$ be an elementary transvection. For $n\geq 3$ the transvection $\bigwedgem{2}t_{i,j}(\xi)$ can be presented as the following product:
$$\bigwedgem{2}t_{i,j}(\xi)=\prod\limits_{k=1}^{i-1} t_{ki,kj}(\xi)\,\cdot\prod\limits_{l=i+1}^{j-1}t_{il,lj}(-\xi)\,\cdot\prod\limits_{m=j+1}^n t_{im,jm}(\xi) \eqno(1)$$
for any $1\leq i<j\leq n$.
\end{proposition}

\begin{remark}
For $i>j$ the similar equality holds:
$$\bigwedgem{2}t_{i,j}(\xi)=\prod\limits_{k=1}^{j-1} t_{ki,kj}(\xi)\,\cdot\prod\limits_{l=j+1}^{i-1}t_{li,jl}(-\xi)\,\cdot\prod\limits_{m=i+1}^n t_{im,jm}(\xi) \eqno(1')$$
\end{remark}

\begin{remark}
A commutator of any two transvections from the right-hand sides of formulas $(1)$ and $(1')$ equals~1. Therefore the commutator with the transvection $\bigwedgem{2}t_{i,j}(\xi)$ is equal to 1 as well.
\end{remark} 

It follows from the proposition that $\bigwedgem{2}t_{i,j}(\xi)\in \E^{n-2}(N,R)$, where a set $\E^M(N,R)$ consists of products of $M$ or less elementary transvections.

Similarly, we can define any exterior powers of the elementary group and calculate exterior powers of elementary transvections (see~\cite{LubNekover1} for details).

\section{Equations}\label{eq}
\noindent

In this Section, we will define the affine group scheme $\bigwedgem{2}\GL_n$ via explicit equations. As mentioned above, it is sufficient to obtain equations on the group of points $\bigwedgem{2}\GL_{n}(\mathbb{C})$.

We will use the fact that the exterior square of the general linear group preserves the Pl\"ucker ideal. We briefly recall a relationship between the exterior square of $\GL_{n}(R)$ and the ideal $\Plu$ generated by the Pl\"ucker relations below. 

The Grassmann space $\Gr_{2, n}(\mathbb{C})$ is embedded in the projective space $\mathbb{P}(\bigwedgem{2}(\mathbb{C}^{n}))\cong \mathbb{P}(\mathbb{C}^{N})$ as a set of decomposable tensors. More precisely, a point of $\Gr_{2, n}(\mathbb{C})$  corresponding to an affine hull of vectors $\{v_{1}, v_{2}\}$ maps to a projective class of an indecomposable tensor $v_{1}\wedge v_{2}$. Under this embedding $\Gr_{2,n}(\mathbb{C})$ is a subvariety of $\mathbb{P}(\mathbb{C}^{N})$. It is defined by the classical Pl\"ucker equations.
Since the group $\GL_{n}(R)$ takes each decomposable tensor to a decomposable one, we see that the group $\bigwedgem{2}{(\GL_{n}(R))}$ preserves an ideal $\Plu=\Plu(n,R) \trianglelefteq R\,[p_{I}:\, I \in \bigwedgem{2}[n]]$, generated by the Pl\"ucker polynomials, where $R\,[p_{I}:\, I \in \bigwedgem{2}[n]]$ is a polynomial algebra with its standard grading by total degree of monomials.

This argument is also true for the general exterior power. But the case of the exterior square is characterized by the following fact. The Pl\"ucker ideal for the exterior square is relatively simple. The next Section is devoted to an explicit description of the structure of the last ideal.

\subsection{Structure of the Pl\"ucker ideal}

Note that the Pl\"ucker polynomials have the following simple form in the case of the exterior square.

The Pl\"ucker polynomials are numbered by pairs $(i, J) \in (\bigwedgem{1}[n]=[n], \bigwedgem{3}[n])$ and equal 
$$f_{i,J}(x) := \sum_{h = 1}^{3}(-1)^{h} \cdot x_{i\,j_{h}}x_{J\backslash\{j_{h}\}},$$
where $x_{ij}$ is the Pl\"ucker coordinates on the projective space $\mathbb{P}(\bigwedgem{2}(\mathbb{C}^{n}))$.

The structure of the Pl\"ucker ideal $\Plu := \langle f_{i,J} \rangle$ is described in the following theorem.

\begin{theorem}\label{Plu}
\begin{enumerate}
\item For any set $1 \leq i < j < k < l \leq n$ denote by $\Plu(i,j,k,l)$ a submodule of the $R$--module $\Plu$:
$$\Plu(i,j,k,l) = \left \langle f_{s,T}: \; s\sqcup T = \{i,j,k,l\} \right \rangle_{R}.$$ 
Then $\Plu(i,j,k,l)$ is one-dimensional. For this space we fix the canonical generator $f_{i, \{j,k,l\}}$. 
\item\label{canon} A set of homogeneous polynomials $\{f_{s,T}\}_{(s,T)}$ is a basis of the ideal $\Plu$ iff for any set $1 \leq i < j < k < l \leq n$ exactly one element $\alpha_{ijkl}\cdot f_{i,\{j,k,l\}}$ is selected, where $\alpha_{ijkl}$ is any invertible constant. In particular, the canonical basis of the ideal $\Plu$ is 
$$\{f_{i, \{j,k,l\}} : 1\leq i<j<k<l \leq n\}.$$
\item\label{condiff} $b = \sum\limits_{I,J} b_{I,J}\,x_{I}x_{J} \in \Plu$ whenever the following conditions are hold:
\begin{enumerate}
\item $b_{I,J}=0$, if the intersection of $I$ and $J$ is not empty;
\item $\sign(I,J)\cdot b_{I,J} =\sign(K,L)\cdot b_{K,L}$, if $I\sqcup J = K\sqcup L$.
\end{enumerate}
\end{enumerate}
\end{theorem}
\begin{proof}
First, note that for any $1 \leq i<j<k<l\leq n$, we have
\begin{align*}
f_{i,j,k,l}(x)& = -x_{ij}x_{kl}+x_{ik}x_{jl}-x_{il}x_{jk};\\
f_{j,i,k,l}(x)& = +x_{ij}x_{kl}+x_{jk}x_{il}-x_{jl}x_{ik} = -f_{i,jkl}(x);\\
f_{k,i,j,l}(x)& = +x_{ik}x_{jl}-x_{jk}x_{il}-x_{kl}x_{ij} = +f_{i,jkl}(x);\\
f_{l,i,j,k}(x)& = +x_{il}x_{jk}-x_{jl}x_{ik}+x_{kl}x_{ij} = -f_{i,jkl}(x).
\end{align*}
Therefore, we proved the first statement and the second one too.

To prove the last one we present any polynomial $b \in \Plu$ in the following form:
$$b = \sum\limits_{I,J}b_{I,J}x_Ix_J=\sum\limits_{\substack{\{i,j,k,l\}\\ i<j<k<l}}\sum\limits_{I\cup J=\{i,j,k,l\}}b_{I,J}x_Ix_J.$$
Using item~\ref{canon}, we get
$$b = \sum_{i<j<k<l} \sum\limits_{I\cup J=\{i,j,k,l\}}b_{I,J}x_Ix_J= \sum_{i<j<k<l} \theta_{\{i,j,k,l\}}(f)\cdot f_{i,jkl}.$$
In other words, $b_{I,J} = \sign(I,J)\cdot \theta_{I \cup J}(f)$. This completes the proof of item~\ref{condiff}.
\end{proof}

\subsection{Equations on group scheme}
In this Section we define the affine group scheme $\bigwedgem{2}\GL_n$ via equations. Let us consider the standard action $g \in \GL_{N}(R)$ on any Pl\"ucker polynomial $f_{i,J}(x)$:
$$g\circ f_{i,J}(x) = f_{i,J}(g(x)) = \sum\limits_{h=1}^3\sum\limits_{a<b}\sum\limits_{c<d}(-1)^h g_{ab,ij_h}g_{cd,J\backslash j_h}x_{ab}x_{cd}.$$
Hence, the condition $g\circ f_{i,J} \in \Plu$ is equivalent to
$$\sum\limits_{h=1}^3\sum\limits_{\substack{a<b\\c<d}}(-1)^hg_{ab,ij_h}g_{cd,J\backslash j_h}x_{ab}x_{cd}=\sum\limits_{h=1}^3\sum\limits_{i'<j'_1<j'_2<j'_3}(-1)^h \theta_{i'\cup J'}\;x_{i',j'_h}x_{J'\backslash j'_h}.$$
That can be rewritten as
$$\sum\limits_{A,C}x_Ax_C\left(\sum\limits_{h=1}^3(-1)^hg_{A,ij_h}g_{C,J\backslash j_h}\right)=\sum\limits_{i'<j'_1<j'_2<j'_3}\sum\limits_{h=1}^3x_{i',j'_h}x_{J'\backslash j'_h}(-1)^h\theta_{i'\cup J'}.$$
Denote by $a_{A,C}^{J} = a_{A,C}^{J}(g)$ the coefficient of the monomial $x_A x_C = x_{C} x_{A}$ in the last equality. Then
\begin{align*}
a_{A,C}^{J}=&\sum\limits_{h=1}^3(-1)^h(g_{A,ij_h}g_{C,J\backslash j_h}+g_{C,ij_h}g_{A,J\backslash j_h}) =\\
=&\sum\limits_{B\sqcup D = H} \sign(B,D) g_{A,B} g_{C,D}\;,
\end{align*}
where the last sum ranges over all unordered partitions of the set $J$ into disjoint pairs.

In general, for a matrix $g \in \GL_{N}(R)$ the numbers 
$$\left\{a_{A,C}^{H} = \sum\limits_{B\sqcup D = H} \sign(B,D) g_{A,B} g_{C,D}\right\}_{A,C,H}$$ 
is called \textbf{exterior numbers of a matrix} $g$.

\begin{theorem}\label{lemmaIFF}
The following conditions are equivalent:
\begin{enumerate}
\item $g \in \bigwedgem{2}\GL_n (R)$.
\item For any $H \in \bigwedgem{4}[n]$ and for any $A,C \in \bigwedgem{2}[n]$,
		\begin{itemize}
		\item if $A\cap C\neq \emptyset$, then $a_{A,C}^{H}(g) = 0$;
		\item if $A\cap C= \emptyset$, then $a_{A,C}^{H}(g) = \sign(A,C)\cdot \theta_{A \cup C}^{H}(g)$, where $\theta_{*}^{\bullet}$ is a function of arguments $*$ and $\bullet$.
\end{itemize}
\end{enumerate}
\end{theorem}

To finish the proof of this theorem, we have to show that $\bigwedgem{2}\GL_n$ coincides with the stabilizer of the ideal $\Plu$.
This fact follows from the classification~\cite[Table 1]{SeitzMaxSub} (examples of such argument can be found in~\cite[Proposition 7]{LubNekover1} and~\cite[Proof of Theorem 2]{VavLuzgE6}).

Let us re-prove the following fact as an example of calculation of exterior numbers.
\begin{theorem}\label{again}
$\bigwedgem{2}\GL_n (R) \geqslant \bigwedgem{2}(\GL_{n}(R)).$
\end{theorem}
\begin{proof}
We need to show that any matrix from the group $\bigwedgem{2}(\GL_{n}(R))$ satisfies the second condition of the last Theorem.

Let $x\in \GL_{n}(R)$ and let $g:=\bigwedgem{2}(x)$. Also let $A=\{a < b\}$, $B=\{c < d\}$, then
$$a_{A,C}^{H}(g)=\sign(A,C)\cdot M_{A\cup C}^H(g), \eqno(*)$$
where $M_{Q}^P$ is a minor of fourth order with rows $Q$ and columns $P$.

In the case $A\cap C\neq \emptyset$, we can assume that $a=c$. Then summands $h_{a,\sigma(i)}h_{b,\sigma(j_1)}h_{a,\sigma(j_2)}h_{d,\sigma(j_3)}$ and $h_{a,\tilde{\sigma}(i)}h_{b,\tilde{\sigma}(j_1)}h_{a,\tilde{\sigma}(j_2)}h_{d,\tilde{\sigma}(j_3)}$  are equal and have different signs, where $\sigma$ and $\tilde{\sigma}$ differ by some permutation ($\sigma(i) = \tilde{\sigma}(j_2),\sigma(j_1) = \tilde{\sigma}(j_3)$). Thus the required sum equals zero.

If $A\cap C= \emptyset$, then for any fixed set $\{a,b,c,d\}$, we have
$$\sum\limits_{\sigma\in S_4}\sign(\sigma)\cdot h_{a,\sigma(i)}h_{b,\sigma(j_1)}h_{c,\sigma(j_2)}h_{d,\sigma(j_3)} = \pm\theta_{A\cup C}^{H}.$$
The sign depends on the permutation $(a,b,c,d)$ and evidently coincides with the required.
\end{proof}

\begin{remark}
Formula $(*)$ and Theorem~\ref{Extens} show that exterior numbers for $x = \bigwedgem{2}(y) \in \bigwedgem{2}\GL_{2}(R)$, where $y \in \GL_{n}(R')$, are minors of fourth order of a matrix $y$. Therefore, the proof of the last Theorem is true in general case: elements of $\bigwedgem{2}(\GL_n(R))$ satisfy these equations. This fact is true due to ``locality'' of the statement: we use only one element, so we can find a suitable extension $R'/R$ for this element.

Since in a general case this extension cannot be chosen for all elements simultaneously, minors could be undefined globally. But exterior numbers define global sections over the whole scheme $\bigwedgem{2}\GL_{n}$. Restrictions of these sections coincide with locally defined minors.
\end{remark}

For description of overgroups it is also necessary to find equations on some congruence--subgroups. Such results are obtained as corollaries below.

Let, as above, $A\trianglelefteq R$, and let $R/A$ be the factor-ring of $R$ modulo $A$. Denote by $\rho_A: R\longrightarrow R/A$ the canonical projection sending $\lambda\in R$ to $\bar{\lambda}=\lambda+A\in R/I$. Applying the projection to all entries of a matrix, we get the reduction homomorphism 
$$\begin{array}{rcl}
\rho_{A}:\GL_{n}(R)&\longrightarrow& \GL_{n}(R/A)\\
a &\mapsto& \overline{a}=(\overline{a}_{i,j})
\end{array}$$

\begin{theorem}\label{Congr}
Let $A$ be an ideal in the ring $R$. In the notation of Theorem~$\ref{lemmaIFF}$ a matrix $g=(g_{A,C})\in \GL_{N}(R)$ belongs to the group $\rho_A^{-1}\big(\bigwedgem{2}\GL_n (R/A)\big)$ whenever the following conditions hold
\begin{center}
\begin{itemize}
\item If $A\cap C\neq \emptyset$, then $a_{A,C}^{H} \equiv 0 \pmod A$;
\item If $A\cap C = \emptyset$, then $a_{A,C}^{H} \equiv \sign(A,C)\cdot \theta_{A\cup C}^{H} \pmod A$.
\end{itemize}
\end{center}
\end{theorem}

\vspace{0.5cm} 

Now we formulate several properties of exterior numbers.

\begin{proposition}
Let $g,h\in \bigwedgem{2}\GL_n (R)$, then 
$$a_{A,C}^H(g\cdot h)=\sign(A,C)\cdot\sum\limits_I\theta_I^H(h)\theta_{A\cup C}^I(g).$$
\end{proposition}

\begin{proof}
The statement follows from the calculation below.
\begin{align*}
a_{A,C}^H(g\cdot h)&=\sum\limits_{B\sqcup D=H}\sign(B,D)\sum\limits_{K,L}g_{A,K}\,h_{K,B}\,g_{C,L}\,h_{L,D}=\sum\limits_{K,L}g_{A,K}\,g_{C,L}\sum\limits_{B\sqcup D=H}\sign(B,D)h_{K,B}\,h_{L,D}=\\
&=\sum\limits_{K,L}g_{A,K}\,g_{C,L}\sign(K,L)\theta_{K\cup L}^H(h)=\sum\limits_I\theta_I^H(h)\sum\limits_{\substack{K,L\\K\cup L=I}}\sign(K,L)g_{A,K}\,g_{C,L}=\\
&=\sum\limits_I\theta_I^H(h)\sign(A,C)\theta_{A\cup C}^I(g).
\end{align*}
\end{proof}

\begin{corollary}
Let $x_i\in \bigwedgem{2}\GL_n(R)$, then
$$a_{A,C}^H\left(\prod\limits_{h=1}^kx_h\right)=\sign(A,C)\cdot \sum\limits_{I_1,\ldots,I_{k-1}}\theta_{I_1}^H(x_k)\theta_{I_2}^{I_1}(x_{k-1})\cdot\ldots\cdot\theta_{I_{k-1}}^{I_{k-2}}(x_2)\theta_{A\cup C}^{I_{k-1}}(x_1),$$
\end{corollary}

Let us give an example of an explicit calculation of the exterior numbers for an element of the exterior square of the elementary group.

\begin{proposition}
\begin{align*}
a_{A,C}^H\left(\bigwedgem{2}(t_{i,j}(\xi))\right)=
\begin{cases}
\sign(A,C),& \text{if } H=A\cup C,\\
\sign(A,C\cup j\backslash i)\cdot\left(\bigwedgem{2}(t_{i,j}(\xi))\right)_{C,C\cup j\backslash i},& \text{if } H\neq A\cup C,\; i\in C,\\
\sign(A\cup j\backslash i,C)\cdot\left(\bigwedgem{2}(t_{i,j}(\xi))\right)_{A,A\cup j\backslash i},& \text{if } H\neq A\cup C,\; i\in A,\\
0,& \text{otherwise}.
\end{cases}
\end{align*}
\end{proposition}
\begin{proof}
It is known that
$a_{A,C}^H\left(\bigwedgem{2}(t_{i,j}(\xi))\right)=\sum\limits_{B\sqcup D=H}\sign(B,D)\left(\bigwedgem{2}(t_{i,j}(\xi))\right)_{A,B}\left(\bigwedgem{2}(t_{i,j}(\xi))\right)_{C,D}.$\\
$\left(\bigwedgem{2}(t_{i,j}(\xi))\right)_{L,M}\neq 0$ whenever either $L=M$ or $L=(ik), M=(jk)$. Let us consider four cases:
\begin{enumerate}
\item $A=B$ and $C=B$, 
\item $A\neq B$ and $C\neq B$ (while $A$ and $C$ contain $i$),
\item $A=B$ and $C\neq D$ (while  $C$ contains $i$),
\item $A\neq B$ and $C=D$ (while $A$ contains $i$).
\end{enumerate}
Let $A=B$ и $C=D$, then $\left(\bigwedgem{2}(t_{i,j}(\xi))\right)_{A,A}=\left(\bigwedgem{2}(t_{i,j}(\xi))\right)_{B,B}=1$.

Now, let $A\neq B$ and $C\neq B$, $i\in A$ and $i\in B$. We can assume that $A=\{i,a\}, C=\{i,c\}$. Then a coefficient $a_{A,C}^H$ has the following form
$$\sum\limits_{B\sqcup D=H}\sign(B,D)\left(\bigwedgem{2}(t_{i,j}(\xi))\right)_{\{i,a\},B}\left(\bigwedgem{2}(t_{i,j}(\xi))\right)_{\{i,c\},D}.$$
Since $B$ and $D$ contain $j$ (otherwise the corresponding transvection element is zero), we have $B=\{j,a\}, D=\{j,c\}$. This is impossible, since intersection of $B$ and $D$ have to be empty. Thus, the second case is excluded. It remains to consider two similar cases. Let $A=B$ and $C\neq D$.\\
The sum can be rewrite in the following form:
\begin{multline*}
\sum\limits_{B\sqcup D=H}\sign(B,D)\left(\bigwedgem{2}(t_{i,j}(\xi))\right)_{A,B}\left(\bigwedgem{2}(t_{i,j}(\xi))\right)_{C,D}=\\
=\sum\limits_{D\subset H\backslash A}\sign(A,D)\left(\bigwedgem{2}(t_{i,j}(\xi))\right)_{A,A}\left(\bigwedgem{2}(t_{i,j}(\xi))\right)_{C,D}=\sum\limits_{D\subset H\backslash A}\sign(A,D)\left(\bigwedgem{2}(t_{i,j}(\xi))\right)_{C,D}.
\end{multline*}
The set $C$ have to contain $i$, hence $C=\{i,c\}$. That $\left(\bigwedgem{2}(t_{i,j}(\xi))\right)_{\{i,c\},D}$ equals not zero (??) it is necessary that $D=\{j,c\}$.  Then the sum equals $\sign(A,C\cup j\backslash i)\cdot(\bigwedgem{2}(t_{i,j}(\xi)))_{C,C\cup j\backslash i}$.\\
The last case is true by the same argument.
\end{proof}

\subsection{Second series of equations}

Alternatively, we will present one more system of equations defining the affine group scheme $\bigwedgem{2}\GL_n$. Note that any Pl\"ucker polynomial $f_{\mathcal{I}} \in \Plu$ is a quadratic form. Then  it can be rewritten as follows 
$$f_{\mathcal{I}}(x)=x^tB_{\mathcal{I}}x,$$
where $B_{\mathcal{I}}$ is a matrix such that its elements equal
$$\left(B_{\mathcal{I}}\right)_{M,L}=\begin{cases}\sign(M,L), &\text{ if } M\sqcup L=\mathcal{I},\\0, &\text{ otherwise.}\end{cases}$$
Then the condition
$$f_{\mathcal{I}}(gx)=g\circ f_{\mathcal{I}}(x)=\sum\limits_{\mathcal{J}}\alpha_{\mathscr{I}}^{\mathcal{I}}f_{\mathscr{I}},\text{ where }\alpha_{\mathscr{I}}^{\mathcal{I}}\in R$$
has the following form
$$x^tg^tB_{\mathcal{I}}gx=\sum\limits_{\mathscr{I}}\alpha_{\mathscr{I}}^\mathcal{I}x^tB_{\mathscr{I}}x \text{ или }g^tB_{\mathcal{I}}=\sum\limits_{\mathscr{I}}\alpha_{\mathscr{I}}^\mathcal{I}B_{\mathscr{I}}g^{-1}.$$
Consequently, we have the useful result.
\begin{theorem} \label{SecondForm}
A matrix $g \in \GL_{N}(R)$ belongs to $\bigwedgem{2}\GL_n(R)$ whenever the identities for some constants $\alpha_{*}^{\bullet}$ hold.\\

$\begin{cases}
\sum\limits_{\substack{N:\\N\cap K=\emptyset}}\alpha_{k\cup N}^{\mathcal{I}}\sign(K,N))g_{N,L}'=0, &\text {for any pairs $K,L$ such that $L\not\subset \mathcal{I}$},\\
\sum\limits_{\substack{N:\\N\cap K=\emptyset}}\alpha_{k\cup N}^{\mathcal{I}}\sign(K,N))g_{N,L}'=g_{\mathcal{I}\backslash L,K}\sign(\mathcal{I}\backslash L,K), &\text {for any pairs $K,L$ such that $L\subset \mathcal{I}.$}
\end{cases}$
\end{theorem}
Note that these equations have more compact form, but its contain a set of undefined constants.

\section{Geometric interpretation of exterior numbers}\label{geom}

In the last Section we present an algorithm for computation of the exterior numbers of any matrix $g\in \bigwedgem{2}\SL_n(R)$ via weight diagrams. We refer the reader to the paper~\cite{atlas}, where the theory of weight diagrams was developed. Also, that paper contains an extensive bibliography.

In Fig.~\ref{Fig1} the weight diagram $(A_3,\varpi_2)$ in the standard basis $\{e_1,e_2,e_3,e_4\}$ is presented.

\centerline{\xymatrix @-1.2pc{
&&{\overset{e_1\wedge e_4}{\bullet}}\ar@{-}[dr]^1 \ar@{-}[dl]_3\\
&{\overset{e_1\wedge e_3}{\bullet}}\ar@{-}[dl]_2 \ar@{-}[dr]_1 &&{\overset{e_2\wedge e_4}{\bullet}}\ar@{-}[dl]^3 \ar@{-}[dr]^2\\
{\overset{e_1\wedge e_2}{\bullet}}&&{\overset{e_2\wedge e_3}{\bullet}}&&{\overset{e_3\wedge e_4}{\bullet}}}}
\captionof{figure}{$(A_3,\varpi_2)$}\label{Fig1}

Recall the definition of the exterior numbers. Let $H:=\{i\}\sqcup J$ be four different numbers from $\{1,\ldots,n\}$. For any matrix $g\in \SL_{N}(R)$ and any pairs of numbers $A,C\in\bigwedgem{2}[n]$:
$$a_{A,C}^H(g)=\sum\limits_{h=1}^3(-1)^h(g_{A,ij_h}g_{C,J\backslash j_h}+g_{C,ij_h}g_{A,J\backslash j_h})=\sum\limits_{B\sqcup D=H}\sign(B,D)g_{A,B}g_{C,D}.$$

Now, we define one more diagram equals the Descartes square of the initial. With its help we can visualize matrix entries of any element of the group $\bigwedgem{2}\SL_n(R)$. To picture this diagram we should construct a ``big'' diagram $(A_{n-1},\varpi_{2})$ and also construct $N$ ``small'' copies of a diagram $(A_{n-1},\varpi_2)$ in all vertices of the ``big'' one.

\centerline{\xymatrix @-1.7pc{
&&&&&&&&&&{\bullet} \ar@{-}[dl] \ar@{-}[dr]\\
&&&&&&&&&{\bullet} \ar@{-}[dl] \ar@{-}[dr] &{\bullet}\ar@{..}[dldldldldldldldl] \ar@{..}[drdrdrdrdrdrdrdr]&{\bullet} \ar@{-}[dl]  \ar@{-}[dr]\\
&&&&&&&&{\bullet}&&{\bullet}&&{\bullet}\\
&&&&&&&&&&&&&&&&&&&&\\
&&&&&&{\bullet} \ar@{-}[dl] \ar@{-}[dr] &&&&&&&&{\bullet} \ar@{-}[dl] \ar@{-}[dr] \\
&&&&&{\bullet} \ar@{-}[dl] \ar@{-}[dr] &{\scriptstyle{13}}\ar@{..}[drdrdrdr]&{\bullet} \ar@{-}[dl] \ar@{-}[dr] &&&&&&{\bullet} \ar@{-}[dl] \ar@{-}[dr] &{\bullet}\ar@{..}[dldldldl]&{\bullet} \ar@{-}[dl] \ar@{-}[dr]\\
&&&&{\bullet}&&{\bullet}&&{\scriptstyle{34}}&&&&{\bullet}&&{\bullet}&&{\bullet}\\
&&&&&&&&&&&&&&&&&&&&\\
&&{\bullet} \ar@{-}[dl] \ar@{-}[dr]&&&&&&&&{\bullet} \ar@{-}[dl] \ar@{-}[dr]&&&&&&&&{\bullet} \ar@{-}[dl] \ar@{-}[dr]\\
&{\bullet} \ar@{-}[dl] \ar@{-}[dr]&{\bullet}&{\bullet} \ar@{-}[dl] \ar@{-}[dr]&&&&&&{\bullet} \ar@{-}[dl] \ar@{-}[dr]&{\bullet}&{\bullet} \ar@{-}[dl] \ar@{-}[dr]&&&&&&{\bullet}\ar@{-}[dl] \ar@{-}[dr]&{\bullet}&{\bullet} \ar@{-}[dl] \ar@{-}[dr]\\
{\bullet}&&{\bullet}&&{\bullet}&&&&{\bullet}&&{\bullet}&&{\bullet}&&&&{\bullet}&&{\bullet}&&{\bullet}}}
\captionof{figure}{$(A_{3},\varpi_2)\times(A_{3},\varpi_2)$}\label{Fig2}

\begin{remark}
In Fig.~\ref{Fig2} a point with coordinates (13, 34) corresponds to an element $g_{13,34}\in \bigwedgem{2}\SL_4(R)$.
\end{remark}

We introduce the notion of a path on a weight diagram of the symmetric square of $\SL_n(R)$ \footnote{i.\,e., a diagram of the exterior square with additional vertices $(i,i)$ for $1\leq i\leq n+1$.}. For a diagram vertex $(ij)\in\bigwedgem{2}[n]$ a path for number $i$ (respectively $j$) is the maximum set of vertices containing $i$ (respectively $j$) and connecting its edges. On the weight diagram for the group $\Sym^2\SL_7(R)$ we draw two paths incident to the vertex $(15)$ ($\sim$ for number $1$, $=$ for number $5$).

\centerline{\xymatrix @-1.7pc {
&&&&&&{\bullet}\ar@{~}[dldldldldl]\ar@{-}[dr]\\
&&&&&{\bullet}\ar@{-}[dr]&&{\bullet}\ar@{-}[dl]\ar@{-}[dr]\\
&&&&{15}\ar@{=}[drdrdrdr]&&{\bullet}\ar@{-}[dl]\ar@{-}[dr]&&{\bullet}\ar@{-}[dl]\ar@{-}[dr]\\
&&&{\bullet}\ar@{-}[dr]&&{\bullet}\ar@{-}[dl]&&{\bullet}\ar@{-}[dl]\ar@{-}[dr]&&{\bullet}\ar@{-}[dl]\ar@{-}[dr]\\
&&{\bullet}\ar@{-}[dr]&&{\bullet}\ar@{-}[dl]\ar@{-}[dr]&&{\bullet}\ar@{-}[dl]&&{\bullet}\ar@{-}[dl]\ar@{-}[dr]&&{\bullet}\ar@{=}[dldl]\ar@{-}[dr]\\
&{\bullet}\ar@{-}[dl]\ar@{-}[dr]&&{\bullet}\ar@{-}[dl]\ar@{-}[dr]&&{\bullet}\ar@{-}[dl]\ar@{-}[dr]&&{\bullet}\ar@{-}[dl]&&{\bullet}\ar@{-}[dr]&&{\bullet}\ar@{-}[dl]\ar@{-}[dr]\\
{\circ}&&{\circ}&&{\circ}&&{\circ}&&{\circ}&&{\circ}&&{\circ}}}

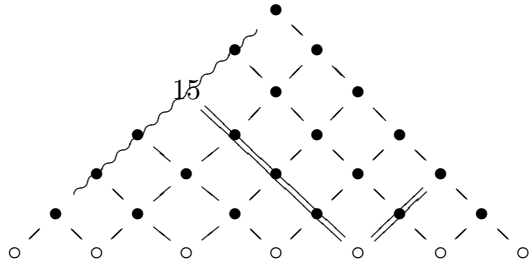
\captionof{figure}{Paths incident to the vertex (15)}\label{Figpath}

A root embedding of a diagram $(A_{3},\varpi_{2})$ in a diagram $(A_{n-1},\varpi_{2})$ is called an elementary square. Notice that any elementary square can be obtained as a result of pairwise intersections of four paths, which are organized into pairs of parallel paths. 

Let us give an example of constructing an elementary square of the bivector representation of the group $\SL_7(R)$ for $H=\{1,2,4,6\}$.
We split the set $H$ as follows: $H=\{14\}\cup \{26\}$. Next we construct four paths incident to these vertices. Points of the intersection of these paths form an elementary square. 

\begin{figure}[h]
\begin{center}
\includegraphics[scale=0.9]{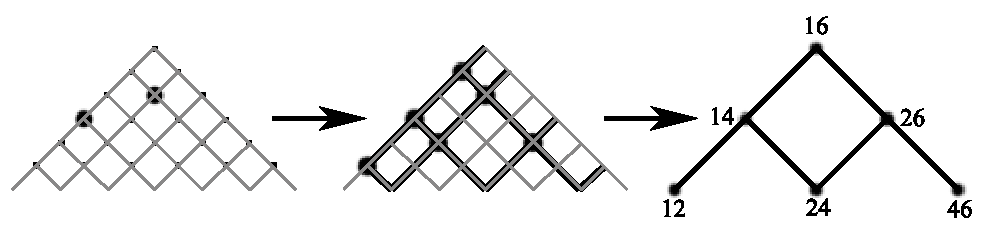}
\caption{Elementary square for the set $\{1,2,4,6\}$}
\label{Figelsq}
\end{center}
\end{figure}

To calculate the exterior number $a_{A,C}^H(g)$ for the fixed set $(A, C, H)$ it is necessary to fix two vertices $A$ and $C$ in the diagram $(A_{n-1}, \varpi_{2})$. Next, we need to consider two copies of a diagram $(A_{n-1}, \varpi_{2})$ corresponding to the fixed vertices $A$ and $C$. Using the set $H$ we construct four paths in these diagrams. The intersection points form elementary squares. Then 
$a_{A,C}^H(g)$ is the  signed sum of all possible pairwise products of vertices of the last two elementary squares. In this sum, the choice of signs for the terms is shown in Fig~\ref{Figsign}.

\centerline{\xymatrix @-1.0pc{
&&{\overset{\ -\ }{\bullet}}\ar@{-}[dr]\ar@{.}@/^/[dd] \ar@{-}[dl]\\
&{\overset{\ +\ }{\bullet}}\ar@{-}[dl]\ar@{.}@/_/[rr] \ar@{-}[dr] &&{\overset{\ +\ }{\bullet}}\ar@{-}[dl] \ar@{-}[dr]\\
{\underset{\ -\ }{\bullet}}\ar@{.}@/_2pc/[rrrr]&&{\underset{\ -\ }{\bullet}}&&{\underset{\text{ }-\text{ }}{\bullet}}}}

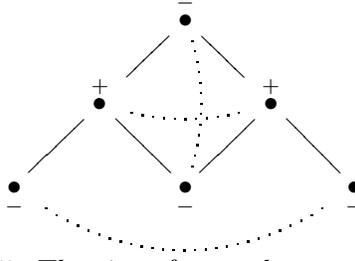
\captionof{figure}{The signs for an elementary square}\label{Figsign}

Let us give an example of calculation of a coefficient $a_{23,24}^{1234}(g)$ for the group $\bigwedgem{2}\SL_4(R)$. Below we spotlight two copies of the diagram $(A_{3}, \varpi_{2})$ corresponding to the indices $A=23$ and $C=24$. Then, 
$$a_{23,24}^{1234}(g)=g_{23,12}\cdot g_{24,34}-g_{23,13}\cdot g_{24,24}+g_{23,14}\cdot g_{24,23}+g_{23,23}\cdot g_{24,14}-g_{23,24}\cdot g_{24,13}+g_{23,34}\cdot g_{24,12}.$$

\begin{figure}[h]
\begin{center}
\includegraphics[scale=0.8]{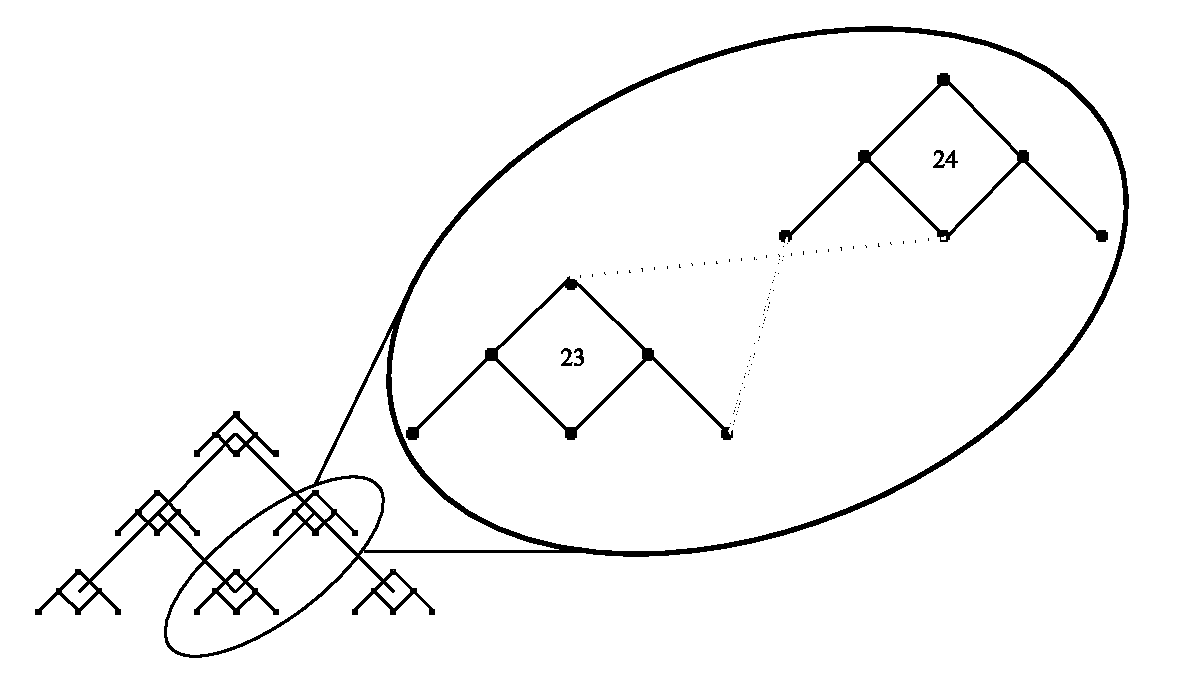}
\caption{$(A_{3}, \varpi_{2})$ subdiagrams for $(A,C) = (23, 24)$}
\label{Fig3}
\end{center}
\end{figure}

\begin{theorem}\label{Tutor}
The following algorithm computes the exterior numbers of any matrix $g\in \bigwedgem{2}\SL_n(R)$.
\end{theorem}

\textbf{4C--Tutorial for computation of $a_{A,C}^H(g)$ for the group $\bigwedgem{2}\SL_n(R)$.}
\begin{itemize}
\item \textbf{C}onstruct a diagram $(A_{n-1},\varpi_{2})\times (A_{n-1},\varpi_{2})$;
\item \textbf{C}hoose two copies of a diagram $(A_{n-1},\varpi_{2})$ corresponding to the numbers $A$ and $C$;
\item \textbf{C}onstruct two elementary squares corresponding to the set $H$ in diagrams from the previous item;
\item \textbf{C}alculate $a_{A,C}^H(g)$ as a signed sum of all possible pairwise products of vertices of the last two elementary squares.
\end{itemize}

\bibliographystyle{zapiski.bst}
\bibliography{english}

\end{document}